\documentclass[12pt, a4 paper,reqno]{amsproc}
\usepackage{epsfig,enumerate,amsmath,amsfonts,amssymb, amsthm, mathrsfs}
\usepackage{hyperref}
\theoremstyle{definition}
\newtheorem{defn}{Definition}
\newtheorem{thm}{Theorem}
\newtheorem{example}{Example}
\newtheorem{lemma}{Lemma}
\newtheorem{remark}{Remark}
\DeclareMathOperator{\RE}{Re}

\begin{document}
\title[On Zeros and Growth of ]{On Zeros and Growth of Solutions of Second Order Linear Differential Equation}
\author[ S. Kumar and M. Saini]{Sanjay Kumar and Manisha Saini} 

\address{Sanjay Kumar \\  Department of Mathematics \\  Deen Dayal Upadhyaya College \\ University of Delhi \\
New Delhi--110 078, India }

\email{sanjpant@gmail.com}

\address{Manisha Saini\\ Department of Mathematics\\ University of Delhi\\ New Delhi--110 007, India}

\email{sainimanisha210@gmail.com }
\thanks {The research work of the second author is supported by research fellowship from University Grants Commission (UGC), New Delhi.}

\begin{abstract}
For a second order linear differential equation $f''+A(z)f'+B(z)f=0$, with $ A(z)$ and $B(z)$ being transcendental entire functions under some restriction, we have established that all non-trivial solutions are of infinite order. In addition, we have proved that these solutions have infinite number of zeros. Also, we have extended these results to higher order linear differential equations.  
\end{abstract}
\maketitle

\section{Introduction}
Consider a second order linear differential equation of the form
\begin{equation}\label{sde}
f''+A(z)f'+B(z)f=0, \quad B(z) \not \equiv 0
\end{equation}
where $A(z)$ and $B(z)$ are entire functions. We have used the notion of Value Distribution Theory of meromorphic function, also known as Nevanlinna Theory \cite{yang}. For an entire function $f$, the order of $f$ and exponent of convergence of $f$ are defined, respectively,  in the following manner,
$$ \rho(f) =  \limsup_{r \rightarrow \infty} \frac{\log^+ \log^+ M(r, f)}{\log r} , \quad \lambda(f) =\limsup_{r\rightarrow \infty} \frac{\log^+ N(r,\frac{1}{f})}{\log r} $$
where $ M(r,f)= \max\{\ |f(z)|:|z| =r \}\ $ is the maximum modulus of $f(z)$ over the circle $|z| =r$ and $N(r,\frac{1}{f})$ is the number of zeros of $f(z)$ enclosed in the disk $|z| <r$. 

It is well known that all solutions of the equation (\ref{sde}) are entire functions. Using Wiman-Valiron theory, it is proved that equation (\ref{sde}) has all solutions of finite order if and only if both $A(z)$ and $B(z)$ are polynomials \cite{lainebook}. Therefore, if either $A(z)$ or $B(z)$ are transcendental entire functions, then almost all solutions of the equation (\ref{sde}) are of infintite order. So, it is natural to find conditions on coeffiicients of the equation (\ref{sde}) such that all non-trivial solutions of the equation (\ref{sde}) are of infinite order. Our aim in this paper is also to find such $A(z)$ and $B(z)$.  It was Gundersen \cite{finitegg}, who gave a necessary condition for equation (\ref{sde}) to have a solution of finite order,   

\begin{thm}
A necessary condition for equation (\ref{sde}) to have a non-trivial solution $f$ of finite order is 
\begin{equation}\label{necc}
\rho(B)\leq \rho(A).
\end{equation} 
\end{thm}
We illustrate this condition with following examples

\begin{example} $f(z)=e^{-z}$ satisfies $f''+e^{-z}f'-(e^{-z}+1)f=0,$ where $\rho(A)=\rho(B)=1$.
\end{example}\label{eg1} 

\begin{example}\label{eg2}
With $A(z)=e^z+2$ and $B(z)=1$ equation (\ref{sde}) has finite order solution  $f(z)=e^{-z}+1$, where $\rho(B)<\rho(A).$
\end{example}
Thus if $\rho(A)<\rho(B)$, then all solutions of the equation (\ref{sde}) are of infinite order. However, given necessary condition is not sufficient, for example

\begin{example}\cite{heitt}
If $A(z)=P(z)e^{z}+Q(z)e^{-z}+R(z)$, where $ P, Q$ and $ R$ are polynomials and $B(z)$ is an entire function with $\rho(B)<1$ then $\rho(f)$ is infinite, for all non-trivial solutions $f$ of the equation (\ref{sde}).
\end{example}
In the same paper \cite{finitegg}, Gundersen proved the following result:
\begin{thm}\label{thm1}
Let $f$ be a non-trivial solution of the equation (\ref{sde}) where either 
\begin{enumerate}[(i)]
\item $\rho(B)< \rho(A)<\frac{1}{2}$\\
or
\item $A(z)$ is transcendental entire function with $\rho(A)=0$ and $B(z)$ is a polynomial
\end{enumerate}
then $\rho(f)$ is infinite.
\end{thm}
Hellerstein, Miles and Rossi \cite{heller} proved Theorem [\ref{thm1}] for $\rho(B)<\rho(A)=\frac{1}{2}.$
In \cite{frei}, Frei showed that the second order differential equation,
\begin{equation}\label{Sde}
f''+e^-f'+B(z)f= 0
\end{equation} 
possesses a solution of finite order if and only if $B(z)=-n^2 , \quad  n \in \mathbb{N}$. Ozawa \cite{ozawa} proved that equation (\ref{Sde}), possesses no solution of finite order when $B(z)=az+b, \quad a \neq0$.  Amemiya and Ozawa \cite{ame}and Gundersen \cite{ggpol} studied the equation (\ref{Sde}) for $B(z)$ being a particular polynomial. After this, Langley \cite{lang} showed that the differential equation 
\begin{equation}
f''+Ce^{-z}f'+B(z)f=0
\end{equation} 
has all non-trivial solutions of infinite order, for any nonzero constant $C$ and for any nonconstant polynomial $B(z)$.
\\

J. R. Long introduced the notion of the deficient value and Borel direction into the studies of the equation (\ref{sde}). For the definition of deficient value, Borel direction and function extremal for Yang's inequality one may refer to \cite{yang}. 
\\

 In \cite{extremal}, J. R. Long proved that if $A(z)$  is an entire function extremal for Yang's inequality and $B(z)$  a transcendental entire function with $\rho(B)\neq \rho(A)$, then all solutions of the equation (\ref{sde}) are of infinite order. In \cite{jlongfab}, J.R. Long replaced the condition $\rho(B) \neq \rho(A)$ with the condition that $B(z) $ is an entire function with \emph{Fabry gaps}.
\\

 X. B. Wu  \cite{wu}, proved that if $A(z)$ is a non-trivial solution of $w''+Q(z)w=0$, where $Q(z)= b_mz^m+\ldots +b_0, \quad b_m \neq 0$ and $B(z)$ be an entire function with $\mu(B)<\frac{1}{2}+ \frac{1}{2(m+1)}$, then all solutions of equation (\ref{sde}) are of infinite order. J.R.  Long \cite{jlongfab} replaced the condition $\mu(B)< \frac{1}{2}+\frac{1}{2(m+1)}$ with $B(z)$ being an entire function with \emph{Fabry gap} such that $\rho(B) \neq \rho(A)$.
\\
The main source of the problems in complex differential equation is Gundersen's \cite{problemgg}. J.R. Long \cite{jlong} gave a partial solution for a question asked by Gundersen in \cite{problemgg}. He proved that: 
\begin{thm}\label{jrthm}
 Let $A(z)=v(z)e^{P(z)}$, where $v(z)( \not \equiv 0)$ is an entire function and $P(z)=a_nz^n +\ldots +a_0$ is a polynomial of degree $n$ such that $\rho(v)<n$. Let $B(z)=b_mz^m +\ldots +b_0$ be a non-constant polynomial of degree $m$, then all non-trivial solutions of the equation (\ref{sde}) have infinite order if one of the following condition holds:
\begin{enumerate}[(i)]
\item $m+2 <2n$;
\item $m+2>2n$ and $m+2 \neq 2kn$ for all integers $k$; 
\item $m+2=2n$ and $\frac{a_n^2}{b_m}$ is not a negative real.
\end{enumerate}
\end{thm}
 In this paper,  we are assuming $B(z)$ to be a transcendental entire function in Theorem \ref{jrthm}. We now recall the notion of \emph{critical rays} and \emph{Fabry gap}:

\begin{defn}\label{def1}\cite{jlong}
Let $P(z)=a_{n}z^n+a_{n-1}z^{n-1}+\ldots +a_0$, $a_n\neq0$ and $\delta(P,\theta)=\RE(a_ne^{\iota n \theta})$. 
A ray $\gamma = re^{\iota \theta}$  is called \emph{critical ray} of $e^{P(z)}$ if $\delta(P,\theta)=0.$ 
\end{defn}

It can be easily seen  that there are $2n$  different critical rays of $e^{P(z)}$ which divides the whole complex plane into $2n$ distinict sectors of equal length $\frac{\pi}{n}.$  Also $\delta(P,\theta)>0$ in $n$ sectors and $\delta(P,\theta)<0$ in remaining $n$ sectors. We note that $\delta(P,\theta)$ is alternatively positive and negative in the $2n$ sectors.
\begin{defn}\cite{hayman} 
Let $g(z)=\sum_{n=0}^{\infty}a_{\lambda_n}z^{\lambda_n}$  be an entire function. If the sequence $(\lambda_n)$ satisfies 
$$ \frac{\lambda_n}{n} \rightarrow \infty $$
as $ n \rightarrow \infty$, then $g(z) $ has Fabry gap. 
\end{defn}
An entire function with Fabry gap has order positive or infinity \cite{hayman}.
We now fix some notations, \\
$ E^+ = \{ \theta \in [0,2\pi]: \delta(P,\theta)\geq 0\}$ and $E^- = \{ \theta \in [0,2\pi]: \delta(P,\theta)\leq 0 \}.$
\\
Let $\alpha>0$ and $\beta>0$ be such that $\alpha<\beta$ then 
\[\Omega(\alpha,\beta)= \{z\in \mathbb{C}: \alpha<\arg z <\beta \}.\]

In this paper, we will prove the following theorem: 
\begin{thm}\label{Main}
Suppose $A(z)=v(z)e^{P(z)}$ be an entire function with $\lambda(A)<\rho(A)=n$, where $P(z)=a_nz^n+ \ldots a_0$ is a polynomial of degree $n$. Suppose that
\begin{enumerate}
\item B(z) be a transcendental entire function satisfying $\rho(B)\neq \rho(A)$ or
\item B(z) be a transcendental entire function with Fabry gap.
\end{enumerate}
 Then all non-trivial solutions of the equation (\ref{sde}) are of infinite order. Moreover, all non-trivial solutions of the equation (\ref{sde}) have infinite number of zeros.  
\end{thm}
In Theorem [\ref{Main}]  part (2), $B(z)$ may be a transcendental entire function with order equal to order of entire function $A(z)$. J. R. Long have proved Theorem [\ref{Main}], for $A(z)$ being an entire function extremal for Yang's inequality in \cite{jlongfab} and \cite{extremal}.
We illustrate our result with some examples
\begin{example}
$$f''+Q(z)e^{P(z)}f'+B(z) f=0,$$ 
where $Q(z)$and $P(z)$ are polynomials and $B(z)$ is any transcendental entire funcion with $\rho(B)\neq $ degree of $P(z)$. Then $\rho(f)=\infty$, for all non-trivial solutions. 
\end{example}
\begin{example}
$$f''+\sin(z) e^{P(z)}f'+ \cos(z^{\frac{n}{2}}) f=0,$$
where $P(z)$ is a polynomial of degree  $m>1, m\neq \frac{n}{2}$ and $n\in \mathbb{N}$ , then all non-trivial solutions are of infinite order.
\end{example}

This paper is orgnised in the following manner: in section 2, we give results which will be useful in proving our main result. In section 3, we will prove our main theorem. In section 4, we will extend our result to higher order linear differential equations.

\section{Auxiliary Result}

In this section, we present some known results, which will be useful in proving Theorem [\ref{Main}]. These results involves logarithmic measure and logarithmic density of sets, therefore we recall these concepts:
\\
 The Lebesgue linear measure of a set $E\subset [0,\infty)$ is defined as $m(E)= \int_{E} dt$. The logarithmic measure of a set $F \subset [1,\infty)$ is given by $m_1(F)= \int_{F}\frac{dt}{t}$. The upper and lower logarithmic densities of  a set $F \subset [1,\infty)$ are given, respectively, by
$$\overline{\log dens(F)} =\limsup_{r\rightarrow \infty}\frac{m_1(F\cap[1,r])}{\log r}$$ 
$$\overline{\log dens(F)} =\liminf_{r\rightarrow \infty}\frac{m_1(F\cap[1,r])}{\log r}$$ 
Also, logarithmic density of a set $F\subset [1,\infty)$ is defined as
 $$\log dens(F)=\overline{\log dens(F)} =\overline{\log dens(F)}.$$
The next lemma is due to Gundersen \cite{log gg} and is used thoroughly during the years.

\begin{lemma}\label{gglog}
Let $f$ be a transcendental entire function of finite order $\rho$, let $\Gamma= \{\ (k_1,j_1), (k_2,j_2) \ldots (k_m,j_m) \}\ $ denote finite set of distinct pairs of integers that satisfy $ k_i > j_ i \geq 0,$  for $i=1,2, \ldots m, $ and let $\epsilon>0$ be a given constant. Then the following three statements holds:

\begin{enumerate}[(i)]
\item  there exists a set $E_1 \subset[0,2\pi)$ that has linear measure zero, such that if $\psi_0 \in [0,2\pi)\setminus E_1, $ then there is a constant $R_0=R_0(\psi_0)>0$  so that for all $z$ satisfying $\arg z =\psi_0$ and $|z| \geq R_0$and for all $(k,j)\in \Gamma$, we have
\begin{equation} \label{guneq}
|f^{(k)}(z)/f^{(j)}(z)| \leq |z|^{(k-j)(\rho-1+\epsilon)}
\end{equation}

\item there exists a set $E_2 \subset (1,\infty)$ that has finite logarithmic measure, such that for all $z$ satisfying $|z| \not \in E_2 \cup [0,1]$  and for all $(k,j) \in \Gamma$, the inequality (\ref{guneq}) holds.

\item there exists a set $E_3\subset [0,\infty)$ that has finite linear measure, such that for all $z$ satisfying $|z|\not \in E_3$ and for all $(k,j) \in \Gamma$, we have
\begin{equation}
|f^{(k)}(z)/ f^{(j)}(z)| \leq |z|^{(k-j)(\rho+\epsilon)}.
\end{equation}
\end{enumerate}
\end{lemma}

The following result gives estimates for absolute value of $ A(z)$ over all complex plane except for a negligible set.
\begin{lemma}\label{implem} \cite{banklang}
Let $A(z)=v(z)e^{P(z)}$ be an entire function with $\lambda(A)<\rho(A)=n$, where $P(z)$ is a polynomial of degree $n$. Then for every $\epsilon>0$ there exists $E \subset [0,2\pi)$ of linear measure zero such that

\begin{enumerate}[(i)]

\item for $ \theta \in E^+ \setminus  E $ there exists $ R>1 $ such that
\begin{equation}
|A(re^{\iota \theta})| \geq \exp \left( (1-\epsilon) \delta(P,\theta)r^n \right)
\end{equation}
for $r>R.$

\item for $\theta \in E^-\setminus E$ there exists $R>1$ such that 
\begin{equation}\label{eq2le}
|A(re^{\iota \theta})| \leq \exp \left( (1-\epsilon)\delta(P,\theta) r^n \right) 
\end{equation}
for $r>R.$
\end{enumerate}
\end{lemma}

Next lemma is from \cite{besi}and give estimates for an entire function of order less than one.
\begin{lemma}\label{gglemma}
Let $w(z)$ be an entire function of order $\rho$, where $0<\rho<\frac{1}{2}$and let $\epsilon>0$ be a given constant. Then there exists a set $S \subset [0,\infty)$ that has upper logarithmic density at least $1-2\rho$ such that $|w(z)| >\exp (|z|^{\rho-\epsilon})$ for all $z$ satisfying $|z| \in S.$ 
\end{lemma}
The following lemma is from \cite{lainebook}.
\begin{lemma}\label{lainebook}
Let $g: (0,\infty) \rightarrow \mathbb{R}$, $h: (0,\infty)\rightarrow \mathbb{R}$ be monotone increasing
functions such that $g(r) < h (r)$ outside of an exceptional set $E$ of finite logarithmic
measure. Then, for any $\alpha > 1$, there exists $r_0 > 0$ such that $g(r) < h(\alpha r)$ holds for all $r > r_0$.
\end{lemma}
Next lemma give property of an entire function with Fabry gap and can be found in \cite{jlongfab}, \cite{zhe}. 
\begin{lemma}\label{fablemma}
Let $g(z)=\sum_{n=0}^{\infty} a_{\lambda_n}z^{\lambda_n}$ be an entire function of finite order with Fabry gap, and $h(z)$ be an entire function with $\rho(h)=\sigma \in (0,\infty)$. Then for any given $\epsilon\in (0,\sigma)$, there exists a set $H\subset (1,+\infty)$ satisfying $ \overline{log dense} H \geq \xi $, where $\xi\in (0,1)$ is a constant  such that for all $|z| =r \in H$, one has
$$ \log M(r,h) > r^{\sigma-\epsilon}, \quad \log m(r,g) > (1-\xi)\log M(r,g),$$

where $M(r,h)=\max \{\ |h(z)|: |z|=r\}\ $, $m(r,g)=\min \{\ |g(z)|: |z|=r\}\ $ and $M(r,g)= \max \{\ |g(z)|: |z|=r\}\ $.
\end{lemma}
The following remark follows from the above lemma.
\begin{remark} \label{fabremark}
Suppsoe that $g(z)=\sum_{n=0}^{\infty} a_{\lambda_n}z^{\lambda_n}$ be an entire function of order $\sigma \in (0,\infty)$ with Fabry gaps then for any given $\epsilon >0, \quad (0<2\epsilon <\sigma)$, there exists a set $H\subset (1,+\infty)$ satisfying $\overline{\log dense}H \geq \xi$, where $\xi \in (0,1)$ is a constant such that for all $|z| =r \in H$ , one has
$$ |g(z)|> M(r,g)^{(1-\xi)}> \exp{\left((1-\xi) r^{\sigma-\epsilon}\right)}>\exp{\left(r^{\sigma-2\epsilon}\right)}.$$
\end{remark}
Next lemma can be found in \cite{lainebook} and can be proved by induction. 
\begin{lemma}\label{ind}
Let $h(z)$ and $Q(z)$ be entire functions and define $f=he^{Q}$. Then $f^{(p)}$ may be represented, for each $p\in \mathbb{N}$, in the form
\begin{equation}\notag
f^{(p)}= \left( h^{(p)}+pQ'h^{(p-1)}+ \sum_{j=2}^{p} \left( {p \choose j} (Q')^j+H_{j-1}(Q') \right) h^{(p-j)} \right)e^Q
\end{equation} 
where $H_{j-1}(Q')$ stands for a differential polynomial of total degree $\leq j-1$ in $Q'$ and its derivatives, with constants coefficients.
\end{lemma} 
We are now able to prove our main result. 
\section{Proof of Theorem \ref{Main}}
This section contains the proof of Theorem [\ref{Main}], which is as follows:
\begin{proof}

If $\rho(A)< \rho(B)$ then by Theorem [\ref{thm1}], all non-trivial solutions $f$ of the equation (\ref{sde}) are of infinite order.  Thus we consider that $\rho(B)\leq\rho(A)< \infty$.

Let us suppose that there exists a non-trivial solution $f$ of the equation (\ref{sde}) such that $\rho(f)<\infty$. Then by Lemma [\ref{gglog}], there exists a set $E_1 \subset[0,2\pi)$ that has linear measure zero, such that if $\psi_0 \in [0,2\pi) \setminus E_1, $ then there is a constant $R_0=R_0(\psi_0)>0$  so that for all $z$ satisfying $\arg z =\psi_0$ and $|z| \geq R_0$, we have

\begin{equation} \label{guneq1}
|f^{(k)}(z)/f(z)|\leq |z| ^{2\rho(f)}, \quad k=1, 2
\end{equation} 
\begin{enumerate}
\item Let $B(z)$ be a transcendental entire function with $\rho(B)\neq \rho(A)$. In this case we need to consider $\rho(B)<\rho(A)$. We consider the following cases on $\rho(B)$.
\begin{enumerate}[(a)]
 \item Suppose that $0<\rho(B)\leq \frac{1}{2}$. Then from Lemma [\ref{gglemma}], there exists a set $S \subset [0,\infty)$ that has upper logarithmic density at least $1-2\rho(B)$ such that 

\begin{equation}\label{eqB}
|B(z)|>\exp(|z| ^{\rho(B)-\epsilon})
\end{equation}
for all $z$, satisfying $|z|\in S.$ 
From equation (\ref{sde}), (\ref{eq2le}), (\ref{guneq1})and (\ref{eqB}), for all $z$, satisfying $\arg z=\psi_0 \in E^- \setminus (E\cup E_1) $and $|z|=r \in S$, $|z|=r >R_0(\psi_0)$ we have

\begin{align*}
\exp{(r ^{\rho(B)-\epsilon})} &< |B(z)| \\
&\leq |f''(z)/f(z)|+ |A(z)| |f'(z)/f(z)| \\
 &\leq  r^{2\rho(f)}(1+o(1))
\end{align*} 
which is a contradiction for arbitrary large $r$.
\item When $\frac{1}{2}\leq \rho(B)< \infty $ then using Phragm$\acute{e}$n- Lindel$\ddot{o}$f principle, there exists a sector $\Omega(\alpha, \beta); \quad 0\leq \alpha<\beta \leq 2\pi$ with $\beta-\alpha \geq \frac{\pi}{\rho(B)}$  such that 

\begin{equation}\label{Border}
\limsup_{r\rightarrow \infty} \frac{\log^+ \log^+ |B(re^{\iota \theta})|}{\log r} =\rho(B)
\end{equation}
for all $\theta \in \Omega(\alpha, \beta)$. Since $\rho(B) < \rho(A)$ this implies that there exists $\theta_0 \in \Omega(\alpha, \beta) \cap \left(E^- \setminus E  \right)$. Thus from equation (\ref{eq2le}) and (\ref{Border}), for $\arg z =\theta_0$ we have,

\begin{equation}\label{eqAle}
|A(re^{\iota \theta_0})| \leq \exp{\left( (1-\epsilon) \delta(P,\theta_0) r^n\right)}
\end{equation}
and
\begin{equation}\label{eqB1}
\exp{\left(r^{\rho(B)-\epsilon}\right)} \leq |B(re^{\iota \theta_0})|
\end{equation}
for sufficiently large $r$. Now from equations (\ref{sde}), (\ref{guneq1}), (\ref{eqAle})and (\ref{eqB1}), for all $z=re^{\iota \theta_0}$, satisfying $\theta_0\in \Omega(\alpha, \beta)\cap E^- \setminus(E\cup E_1)$ and $|z|=r > R_0(\theta_0) $ we have,
\begin{align*}
\exp{(r ^{\rho(B)-\epsilon})} &< |B(z)| \\
&\leq |f''(z)/f(z)| + |A(z)| | f'(z)/f(z)|\\
 &\leq  r^{2\rho(f)}(1+o(1))
\end{align*} 
which is a contradiction for arbitrary large $r$.

\item Now suppose that $B(z)$ is a transcendental entire function with $\rho(B)=0$, then using a result from \cite{pd}, for all $\theta \in [0,2\pi)$ one has,
\begin{equation}\label{eqB2}
\limsup_{r\rightarrow \infty} \frac{\log |B(re^{\iota \theta})|}{\log r} =\infty
\end{equation} 
this implies that for any large $G>0$ there exists $R(G)>0$ such that 
\begin{equation}\label{eqB3}
r^G \leq |B(re^{\iota \theta})|
\end{equation}
for all $\theta \in [0,2\pi)$ and for all $r>R(G)$. From equations (\ref{sde}), (\ref{eq2le}), (\ref{guneq1})and (\ref{eqB3}), for all $z=re^{\iota \theta}$ satisfying $\arg z=\theta \in E^-\setminus \left( E \cup E_1  \right )$ and $|z| =r >R$ we have,
\begin{align*}
r ^G &< |B(z)| \\
&\leq |f''(z)/f(z)|+ |A(z)| |f'(z)/f(z)| \\
 &\leq  r^{2\rho(f)}(1+o(1))
\end{align*} 
which is a contradiction for arbitrary large $r$.
\end{enumerate}
Thus all non-trivial solutions of the equation (\ref{sde}) are of infinite order in this case.
\item Let $B(z)$ be a transcendental entire function with Fabry gap. Then from Lemma (\ref{fablemma}), for any given $\epsilon >0, \quad (0<2\epsilon <\rho(B))$, there exists a set $H\subset (1,+\infty)$ satisfying $\overline{\log dense}H \geq \xi$, where $\xi \in (0,1)$ is a constant such that for all $|z| =r \in H$, one has
\begin{equation}\label{eq2B}
|B(z)| >\exp{\left(r^{\rho(B)-2\epsilon}\right)}
\end{equation}
From equation (\ref{sde}), (\ref{eq2le}), (\ref{guneq1})and (\ref{eq2B}), for all $z$ satisfying $\arg z = \psi_0 \in E^- \setminus (E\cup E_1)$and $|z|=r \in H, \quad  r >R_0(\psi_0) $, we have
\begin{align*}
\exp{\left( r^{\rho(B)-2\epsilon}\right)}&< |B(z)| \leq |f''(z)/f(z)| + |A(z)| |f'(z)/f(z)| \\
 &\leq  r^{2\rho(f)}(1+o(1))
\end{align*}
which is a contradiction for arbitrary large $r$.
\end{enumerate}
We thus conclude that all non-trivial solutions of the equation (\ref{sde}) are of infinite order.

Now let us suppose that $f(z)=h(z)e^{Q(z)}$, where $h(z)$and $Q(z)$ are entire functions, be a non-trivial solution of the equation (\ref{sde}) and hence $\rho(f)=\infty$. 

First we suppose that $\lambda(f)=\rho(h)< \rho(f)$. From equation (\ref{sde}), we have
\begin{equation}
h''+ \left( A(z)+2 Q'(z) \right) h+ \left( B(z) +Q''(z) +(Q')^2(z) \right)=0
\end{equation}
which implies that $\rho(h)\geq \max \{\ \rho(A), \rho(B), \rho(Q) \}\ >0$. As a consequence of this, $f$ contains infinite number of zeros.

If we suppose that $\lambda(f)=\rho(f)=\infty$ then it is clear that $f$ has infinite number of zeros.
\end{proof}\section{Further Results}
In this section we will extend our result to higher order linear differential equations. We consider the higher order linear differential equation as follows:
\begin{equation}\label{sde1}
f^{(m)}+A_{(m-1)}(z)f^{(m-1)}+\ldots +A_1(z)f'+A_0(z)f=0 
\end{equation}
where $m\geq 2$ and $A_0, A_1, \ldots, A_{(m-1)}$ are entire functions. Then it is well known that  all solutions of the equation (\ref{sde1}) are entire functions. Moreover, if $A_0, A_1, \ldots, A_{(m-1)}$ are polynomials then  all solutions of the equation (\ref{sde1}) are of finite orde and vice-versa \cite{lainebook}. Therefore, if any of the coefficient is a transcendental entire function then equation (\ref{sde1}) will possesses almost all solutions of infinite order. However, conditions on coefficients of the equation (\ref{sde1}) are found so that all solutions are of infinite order \cite{lainebook}. Here, we are giving one of such condition on the coefficients of the equation (\ref{sde1}).
\begin{thm}
Suppose there exist an integer $j \in \{\ 1,2, \ldots , m-1 \}\ $ such that $ \lambda(A_j)<\rho(A_j)$. Suppose that $A_0$ be a transcendental entire function satisfying $ \rho(A_i) <\rho(A_0)$ where $i=1,2,\ldots m-1, i\neq j$ with 
\begin{enumerate}
\item $\rho(A_0)\neq \rho(A_j)$ or
\item $A_0(z)$ be a transcendental entire function with Fabry gap.
\end{enumerate}

Then every non-trivial solution of the equation (\ref{sde1}) is of infinite order. In addition, all non-trivial solutions of the equation (\ref{sde1}) has infinite number of zeros.
\end{thm}
\begin{proof}

First let us suppose that $\rho(A_j)<\rho(A_0)$. Then suppose that there exist a solution $f \not \equiv 0$ of the equation (\ref{sde1}) such that $\rho(f) < \infty$, then by Lemma [\ref{gglog}] (ii) there exists a set $ E_2 \subset (1,\infty)$ that has finite logarithmic measure, such that for all $z$ satisfying $|z|\not \in E_2\cup[0,1]$  such that
\begin{equation} \label{guneqq}
|f^{(k)}(z)/f(z)| \leq |z|^{m \rho(f)}
\end{equation}
where $k=1,2,\ldots, m$. Using equation (\ref{sde1}) and (\ref{guneqq}), we have
\begin{align*}
|A_0(z)|&\leq \big|\frac{f^{(m)}(z)}{f(z)}\big| +|A_{(m-1)}(z)|\big|\frac{f^{(m-1)}(z)}{f(z)}\big|+\ldots +|A_1(z)|\big|\frac{f'(z)}{f(z)}\big| \\
&\leq |z|^{m\rho(f)} \{\ 1+|A_{(m-1)}(z)|+\dots +|A_1(z)| \}\
\end{align*}
for all $z$ satisfying $|z| \not \in E_2\cup[0,1]$. From here we get that 
\begin{equation}
T(r,A_0)\leq m\rho(f) \log r+(m-1) T(r,A_i) +O(1)
\end{equation}
where $T(r,A_i)=\max \{\ T(r,A_k): k=1,2,\ldots, m-1\}\ $ and $|z|=r \not \in E_2\cup[0,1]$. Using Lemma [\ref{lainebook}], this implies that $\rho(A_0) \leq \rho(A_i)$, which is a contradiction. Thus all non-trivial solutions of the equation (\ref{sde1}) are of infinite order in this case.

Now consider $\rho(A_0) \leq \rho(A_j)$ and there exists a non-trivial solution $f$ of finite order then by Lemma [\ref{gglog}] (i), there exists a set $E_1 \subset [0,2 \pi)$ with linear measure zero such that if $\psi_0 \in [0,2\pi)\setminus E_1, $ then there is a constant $R_0=R_0(\psi_0)>0$  so that for all $z$ satisfying $\arg z =\psi_0$ and $|z| \geq R_0$ we have
\begin{equation} \label{guneqqq}
|f^{(k)}(z)/f(z)| \leq |z|^{m\rho(f)} \qquad k=1,2, \ldots, m-1
\end{equation}  
Since $\rho(A_i)<\rho(A_0)$, for all $i=1,2, \ldots, m-1, i\neq j$ then for any constant $\eta>0$ such that \\
$ \max \{\ \rho(A_i): i=1,2, \ldots m-1, i\neq j \}\ <\eta< \rho(A_0)$ there exists $R_0>0$ such that 
\begin{equation}\label{Aieq}
|A_i(z)|\leq \exp{|z|^\eta}
\end{equation}
where $i=1,2,\ldots, m-1, i\neq j$ and $|z|=r >R_0$. \\
Also $\lambda(A_j)<\rho(A_j)=n$ then $A_j(z)=v(z)e^{P(z)}$, where $v(z)$ is an entire function and $P(z)$ is a polynomial of degree $n$. 
\begin{enumerate}
\item Let $A_0(z)$ be a transcendental entire function with $\rho(A_0)\neq \rho(A_j)$. In this case we need to consider that $\rho(A_0)<\rho(A_j)$. We will discuss following three cases:
\begin{enumerate}[(a)]
\item\label{casea} suppose $0<\rho(A_0)<\frac{1}{2}$ then by Lemma [\ref{gglemma}], for $0<\epsilon< (\rho(A_0)-\eta)$ there exists a set $S \subset [0,\infty)$ that has upper logarithmic density at least $1-2\rho(A_0)$ such that 
\begin{equation}\label{A0eq}
|A_0(z)| >\exp (|z|^{\rho(A_0)-\epsilon})
\end{equation} for all $z$ satisfying $|z| \in S.$ Now using equation (\ref{eq2le}), (\ref{sde1}), (\ref{guneqqq}), (\ref{Aieq})and (\ref{A0eq}) we have
\begin{align*}
\qquad \qquad \exp{ (|z|^{\rho(A_0)-\epsilon})} &<|A_0(z)| \\
&\leq \big|\frac{f^{(m)}(z)}{f(z)}\big| +|A_{(m-1)}(z)|\big|\frac{f^{(m-1)}(z)}{f(z)}\big| \\
&+\ldots +|A_1(z)|\big|\frac{f'(z)}{f(z)}\big| \\
&\leq r^{m\rho(f)} \{\ 1+ \exp{r^\eta}\\
&+ \ldots + \exp \left( (1-\epsilon) \delta(P,\psi_0)r^n \right)+ \ldots +  \exp{r^\eta} \}\ \\
&= r^{m\rho(f)}\{\ 1+ (m-2) \exp{r^\eta} + o(1)\}\ 
\end{align*}
for all $z$ satisfying $|z|=r \in S$ and $\arg z =\psi_0 \in E^- \setminus (E\cup E_1)$. From here we will get contradiction for sufficiently large $r$.
\item Now suppose that $\rho(A_0) \geq \frac{1}{2}$, then using Phragm$\acute{e}$n- Lindel$\ddot{o}$f principle, there exists a sector $\Omega(\alpha, \beta); \quad 0\leq \alpha<\beta \leq 2\pi$ with $\beta-\alpha \geq \frac{\pi}{\rho(A_0)}$  such that 

\begin{equation}\label{A0order}
\limsup_{r\rightarrow \infty} \frac{\log^+ \log^+ |A_0(re^{\iota \theta})|}{\log r} =\rho(A_0)
\end{equation}
for all $\theta \in \Omega(\alpha, \beta)$. Since $\rho(A_0) < \rho(A_j)$ this implies that there exists $\theta_0 \in \Omega(\alpha, \beta) \cap \left(E^- \setminus E  \right)$ such that 
\begin{equation}\label{Ajeq}
|A_j(re^{\iota \theta_0})| \leq \exp{\left( (1-\epsilon)\delta(P,\theta_0)r^n \right)}
\end{equation} 
and form equation (\ref{A0order}), we have
\begin{equation}\label{Aoeq}
|A_0(re^{\iota \theta_0})| \geq \exp{r^{\rho(A_0)-\epsilon}}
\end{equation}
Thus we get contradiction using equation (\ref{sde1}), (\ref{guneqqq}), (\ref{Aieq}), (\ref{Ajeq})and (\ref{Aoeq}) for sufficiently large $r$ by using similar argument as in case (\ref{casea}).
\item Suppose $A_0$ be a transcendental entire function with $\rho(A_0)=0$, then using a result from \cite{pd}, for all $\theta \in [0,2\pi)$ one has,
\begin{equation}\label{eqA02}
\limsup_{r\rightarrow \infty} \frac{\log |A_0(re^{\iota \theta})|}{\log r} =\infty
\end{equation} 
this implies that for any large $G>0$ there exists $R(G)>0$ such that 
\begin{equation}\label{eqA03}
r^G \leq |A_0(re^{\iota \theta})|
\end{equation}
for all $\theta \in [0,2\pi)$ and for all $r>R(G)$. From equations (\ref{eq2le}), (\ref{sde1}), (\ref{guneqqq}), (\ref{Aieq})and (\ref{eqA03}) we get a contradiction for sufficiently large $r$ using similar argument as in case (\ref{casea}).\\
Thus we conclude that all non-trivial solutions of the equation (\ref{sde1}) are of infinite order in this case.
\end{enumerate}
\item Suppose that $A_0(z)$ be a trascendental entire function with Fabry gap then using Lemma [\ref{fablemma}], for any given $\epsilon >0, \quad (0<2\epsilon <\rho(A_0)-\eta)$, there exists a set $H\subset (1,+\infty)$ satisfying $\overline{\log dense}H \geq \xi$, where $\xi \in (0,1)$ is a constant such that for all $|z|=r \in H$ , one has
\begin{equation}\label{eq2A0}
 |B(z)| >\exp{\left(r^{\rho(A_0)-2\epsilon}\right)}
\end{equation}
From equation (\ref{eq2le}), (\ref{sde1}), (\ref{guneqqq}), (\ref{Aieq})and (\ref{eq2A0}), for all $z$ satisfying $\arg z = \psi_0 \in E^- \setminus (E\cup E_1)$and $|z| =r \in H$, $r>R_0(\psi_0) $, we have
\begin{align*}
\qquad \qquad \exp{\left( r^{\rho(A_0)-2\epsilon}\right)}&< |A_0(z)| \\
&\leq  \big|\frac{f^{(m)}(z)}{f(z)}\big| +|A_{(m-1)}(z)|\big|\frac{f^{(m-1)}(z)}{f(z)}\big| \\
&+\ldots +|A_1(z)|\big|\frac{f'(z)}{f(z)}\big| \\
&\leq r^{m\rho(f)} \{\ 1+ \exp{r^\eta} \\
&+ \ldots + \exp \left( (1-\epsilon) \delta(P,\psi_0)r^n \right) \\
&+ \ldots +  \exp{r^\eta} \}\ \\
&= r^{m\rho(f)} \{\ 1+ (m-2) \exp{r^\eta+o(1)} \}\ 
\end{align*}
which is a contradiction for arbitrary large $r$.
\end{enumerate}
Thus all solutions $f\not \equiv 0$ of the equation (\ref{sde1}) are of infinite order.
Suppose that $f(z)=h(z)e^{Q(z)}$, where $h(z)$and $Q(z)$ are entire functions, be a non-trivial solution of the equation (\ref{sde1}) therefore $\rho(f)=\infty$. \\
Let us suppose that $\rho(h)=\lambda(f)<\rho(f)$. Now from equation (\ref{sde1}) and Lemma [\ref{ind}] we get
\begin{equation}\label{indu}
h^{m}+B_{m-1}(z)h^{(m-1)}+\ldots+ B_0(z)h=0 
\end{equation}
where $$ B_{m-1}=A_{m-1}+mQ'$$
\begin{align*}
B_{m-j}&=A_{m-j}+(m-j+1)A_{m-j+1}Q' \\
&+ \sum_{i=2}^{j} \left( {m-j+1 \choose i}(Q')^i+H_{i-1}(Q') \right) A_{m-j+i}.
\end{align*}
where $j=2,3, \ldots, m$ and $A_m \equiv 1$. Equation (\ref{indu}) implies that $\rho(h)\geq \max \{\ \rho(Q), \rho(A_0),  \rho(A_1), \ldots, \rho(A_{m-1} \}\ >0$. Thus $f(z)$ has infinite number of zero.\\
If $\lambda(f)=\rho(f)=\infty$ then also zeros of $f(z)$ are infinite.
\end{proof}

\end{document}